\documentclass[letterpaper, 10 pt, conference]{ieeeconf}  

\IEEEoverridecommandlockouts                              
\usepackage{amsmath,amssymb,enumerate,subfigure,multirow,hhline,color}
\usepackage{xcolor}
\usepackage{hyperref}
\usepackage{graphicx}
\usepackage{graphics}
\usepackage[sort]{cite}
\usepackage{algorithm}
\usepackage{algpseudocode}
\usepackage{tcolorbox}
\usepackage{epsfig,psfrag}
\usepackage{tabularx}
\usepackage{tabulary}
\usepackage[sort]{cite}

\usepackage{hyperref}
\hypersetup{breaklinks=true,colorlinks=true,linkcolor=blue,citecolor=blue}

\usepackage{tcolorbox}

\usepackage[noabbrev,capitalise,nameinlink]{cleveref}

\DeclareMathOperator{\arginf}{arg\,inf}

\newtheorem{theorem}{Theorem}

\newtheorem{assumption}{Assumption}
\newtheorem{remark}{Remark}
\newtheorem{corollary}{Corollary}
\newtheorem{problem}{Problem}
\newtheorem{lemma}{Lemma}
\newtheorem{definition}{Definition}
\newtheorem{proposition}{Proposition}

\allowdisplaybreaks

\title{Dynamic Programming for POMDP with Jointly Discrete and Continuous State-Spaces}

\author{Donghwan Lee, Niao He, and Jianghai Hu
\thanks{This material is based upon work supported by the National Science Foundation under Grant No. 1539527}
\thanks{D. Lee and N. He are with Coordinated Science Laboratory (CSL),
University of Illinois at Urbana-Champaign, IL 61801, USA {\tt\small
donghwan@illinois.edu}, {\tt\small niaohe@illinois.edu}.}
\thanks{J. Hu is with the Department of Electrical and Computer Engineering,
Purdue University, West Lafayette, IN 47906, USA {\tt\small
jianghai@purdue.edu}.} }
\begin{document}
\maketitle

\begin{abstract}
In this work, we study dynamic programming (DP)
algorithms for partially observable Markov decision processes with
jointly continuous and discrete state-spaces. We consider a class of stochastic systems which have coupled discrete and continuous systems, where only the continuous state
is observable. Such a family of systems includes many real-world systems, for example, Markovian jump linear systems and physical systems interacting with humans. A finite history of observations is used as a
new information state, and the convergence of the corresponding DP
algorithms is proved. In particular, we prove that the DP iterations
converge to a certain bounded set around an optimal solution. Although deterministic DP algorithms are studied in this paper, it is expected that this fundamental work lays foundations for advanced studies on reinforcement learning algorithms under the same family of systems.
\end{abstract}

\section{Introduction}

The goal of this paper is to study optimal control problems for
partially observable Markov decision processes (POMDPs) with
jointly continuous and discrete state-spaces. Optimal control
designs for stochastic systems have been a fundamental research
field for a long time~\cite{bertsekas1995dynamic}. Classical and
popular approaches include, for example, the linear quadratic
Gaussian control and stochastic model predictive control, where the stochastic model predictive control computes a suboptimal control policy with predictions of finite future trajectories. In this paper, we focus on stochastic systems with a special structure where the continuous state-space and discrete
state-space coexist and interact with each other. The state in the
discrete state-space evolves according to a Markov chain which
depends on the state of the control system with the continuous state-space. The overall system can
be viewed as a Markov decision process (MDP)~\cite{bertsekas1995dynamic} with coupled continuous and discrete
state spaces. Such classes of systems arise in several stochastic control
applications including
\begin{itemize}
\item vehicle path-planning~\cite{blackmore2007robust,blackmore2010probabilistic};

\item Markovian jump linear systems with examples of macroeconomic model~\cite{patrinos2014stochastic} and economic models of government expenditure~\cite{costa1999constrained};

\item vehicle controls with driver's behavior
models~\cite{bichi2010stochastic,di2014stochastic};

\item building climate control with human interactions~\cite{page2008generalised,dobbs2014model,donghwan2018simulation,donghwan2018approximate};

\item hybrid electric vehicle powertrain management~\cite{kolmanovsky2008discrete,johannesson2007assessing};
\end{itemize}
to name just a few.

The main contribution of this paper is the first
formal study of dynamic programming formulation and its
convergence results for the systems mentioned above with an additional assumption of the partial state observability, i.e., the discrete state is unobservable. An example includes control systems with human interactions, where unobservable human cognition and behaviors are modelled as discrete Markov chains or Markov decision processes. In particular, this class of systems arises in building management problems with human interactions~\cite{donghwan2018approximate}. To improve the control performance, a finite observation history is considered for the output-feedback control policy with its performance analysis. The use of finite observation histories for POMDPs is a common practice in the reinforcement learning literature~\cite{mnih2015human,hausknecht2015deep}. However, to our knowledge, there exists no attempt to analyze its sub-optimality so far. The proposed results build upon the previous work~\cite{donghwan2018approximate},~\cite{kolmanovsky2008discrete}. Although deterministic DP algorithms are studied in this paper, this fundamental work lays foundations for advanced studies on reinforcement learning algorithms under the same family of systems.

\subsection{Related Work}
As long as a sufficient number of random samples of states can be obtained, the scenario-based (or sample-based) approximation approach, e.g.,~\cite{prandini2000probabilistic,blackmore2007robust,blackmore2010probabilistic,calafiore2013stochastic,schildbach2014scenario,calafiore2013robust}, can design a control policy for complex stochastic systems with generic
probability distributions of uncertainties. It was successfully applied to
robot path-planning problems in~\cite{blackmore2007robust,blackmore2010probabilistic} and
aircraft conflict detection in~\cite{prandini2000probabilistic}.
In the scenario-based approach, the process related to uncertainties and the evolution of the control system trajectories are not fully coupled. For instance, in~\cite{blackmore2007robust,blackmore2010probabilistic}, the uncertainties of the continuous system are affected by the discrete system's state, while the latter does not depend on the continuous system's state.

Fully coupled systems similar to those in this paper
were studied in~\cite{kolmanovsky2008discrete,johannesson2007assessing} for
hybrid electric vehicle powertrain management problems. They are fully coupled in the sense that both continuous and discrete system states affect each other. They
considered approximate dynamic programming~\cite{bertsekas1996neuro} (or reinforcement learning~\cite{sutton1998reinforcement} from the machine learning
context). Compared to~\cite{kolmanovsky2008discrete}, the control
systems in this paper have additional continuous stochastic
disturbances, while~\cite{kolmanovsky2008discrete} only addresses
discrete stochastic disturbances that depend on the discrete state
of the Markov chain. Reinforcement learning algorithms were adopted in~\cite{donghwan2018simulation,donghwan2018approximate} for
building management systems with occupant interactions. Compared to~\cite{donghwan2018approximate}, we consider more general Markov chain models for the discrete state-space
evolution.

\section{Preliminaries and Problem Formulation}\label{section:preliminaries and problem formulation}

\subsection{Noataion}
Throughout the paper, the following notations will be used:
${\mathbb N}$ and ${\mathbb N}_+$: sets of nonnegative and
positive integers, respectively; ${\mathbb R}$: set of real
numbers; ${\mathbb R}^n $: $n$-dimensional Euclidean space;
${\mathbb R}^{n \times m}$: set of all $n \times m$ real matrices;
$A^T$: transpose of matrix~$A$; ${\mathbb S} ^n $ (resp. ${\mathbb
S}_+^n$, ${\mathbb S}_{++}^n$): set of symmetric (resp. positive
semi-definite, positive definite) $n\times n$ matrices; $|S|$:
cardinality of a finite set $S$; ${\mathbb E}[\cdot]$: expectation
operator; ${\mathbb P}[\cdot]$: probability of an event; ${\rm
diam}(C)$: diameter of a set $C$ in ${\mathbb R}^n$, i.e., ${\rm
diam}(C):=\sup\{\| s-s'\|_2 :s,s' \in C\}$; for any vector $x$,
$[x]_i$ is its $i$-th element; for any matrix $P$, $[P]_{ij}$
indicates its element in $i$-th row and $j$-th column; matrix $P
\in {\mathbb R}^{n \times n}$ is called a (row) stochastic matrix
if its row sums are one; $x \in {\mathbb R}^n$ is called a
stochastic vector if its column sum is one; if ${\bf z}$ is a
discrete random variable which has $n$ values and $\mu \in
{\mathbb R}^n$ is a stochastic vector, then ${\bf z} \sim \mu$
stands for ${\mathbb P}[{\bf z} = i] = [\mu]_i$ for all $i \in
\{1,\ldots,n \}$; if ${\bf z}$ is a continuous random variable
with the density function $\rho(\cdot)$, we denote ${\bf z} \sim
\rho(\cdot)$; $\Delta_k$ with $k \in {\mathbb N}_+$: unit simplex
defined as $\Delta_k:=\{ (\alpha_1,\ldots,\alpha_k
):\alpha_1+\cdots+\alpha_k =1,\alpha_i\ge 0,\forall i\in \{
1,\ldots,k\} \}$; w.r.t: abbreviation for ``with respect to.''
Throughout the paper, random variables will be highlighted by
boldface fonts, while the corresponding realizations will be
written by plain fonts.

\subsection{Markov Decision Process}

In this paper, we consider a discrete-time Markov decision process
(MDP)~\cite{bertsekas1996neuro} defined as a tuple $\langle
X,S,U,p_{\bf x},p_{\bf s},p_{\bf r},{\bf r},\gamma\rangle$, where $X$ is a continuous compact state-space,
$S$ is a discrete finite state-space, $U$ is a continuous or
discrete action-space, $p_{\bf x}(x'|x,s,u)$ defines a
continuous state transition probability density function from the
current state $x\in X$ to the next state $x'\in X$ under the action
$u\in U$, and $s\in S$, $p_{\bf s}(s'|s,x)$ defines the discrete
state transition probability mass function from the current state
$s\in S$ to the next state $s'\in S$ under the current
continuous state $x$, ${\bf r}:X \times S \times U \to {\mathbb
R}$ is a stochastic reward function with its expectation
${\mathbb E}[{\bf r}(x,s,u)]= R(x,s,u)$ and density function
$p_{\bf r}({\bf r}|x,s,u)$ given $(x,s,u)$, and $\gamma\in [0,1)$
is called the discount factor. We assume that the expected reward is
bounded. For simplicity, we only consider the reward function
${\bf r}:X\times U \to {\mathbb R}$ which depends on the
continuous state $x$ and action $u$.
\begin{assumption}\label{assumption0}
The expected reward $R$ satisfies $R(x,s)\in [0,M]$ for all $(x,s)
\in X \times U$, where $M \in {\mathbb R}_{++}$.
\end{assumption}

The overall system can be expressed as
\begin{align}
&\begin{cases}
 {\bf x}(k+1)\sim p_{\bf x}(\cdot |{\bf x}(k),{\bf s}(k),{\bf u}(k)),\quad {\bf x}(0)\sim \rho_{\bf x}(\cdot) \\
 {\bf s}(k+1) \sim p_{\bf s} (\cdot|{\bf s}(k),{\bf x}(k)),\quad {\bf s}(0)\sim\rho_{\bf s}(\cdot)\\
 \end{cases}\label{eq:stochastic-system}
\end{align}
where $k\in {\mathbb N}$ is the time step, ${\bf x}(k)\in X$ is
the continuous state at time $k$, ${\bf u}(k)\in U$ is the
control input, and ${\bf s}(k)\in S$ is the discrete state,
$\rho_{\bf x}$ is the initial distribution of ${\bf x}(0)$, and
$\rho_{\bf s}$ is the initial distribution of ${\bf s}(0)$. In
this paper, the discrete state transition $p_{\bf s}(s'|s,x)$ is
represented by a Markov chain with and the state transition matrix
$P(x)$ parameterized by $x \in X$. A visual description of the
system is given in~\cref{fig:stochastic-system}. 
\begin{figure}[h!]
\centering\epsfig{figure=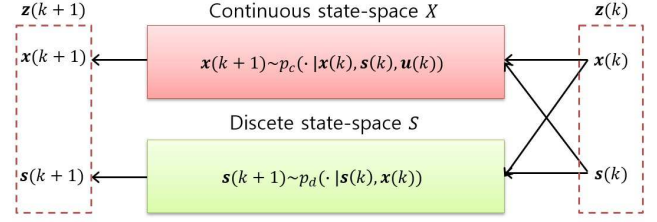,width=8.5cm}
\caption{Overall MDP structure.}\label{fig:stochastic-system}
\end{figure}
If we define the augmented state ${\bf z}(k):=\begin{bmatrix} {\bf x}(k)^T & {\bf s}(k)^T \\ \end{bmatrix}^T$, then~\eqref{eq:stochastic-system} can be formulated by the single Markov decision
process (MDP) with the transition probability density $p_{\bf z}(z'|z,u)$
\begin{align}
&{\bf z}(k+1)\sim p_{\bf z}(\cdot|{\bf z}(k),{\bf u}(k)),\quad {\bf z}(0)\sim \rho_{\bf z}(\cdot)\label{eq:stochastic-system2}
\end{align}
where the continuous and discrete state-spaces coexist and interact with each other.

\subsection{Information State}
\begin{figure}[h!]
\centering\epsfig{figure=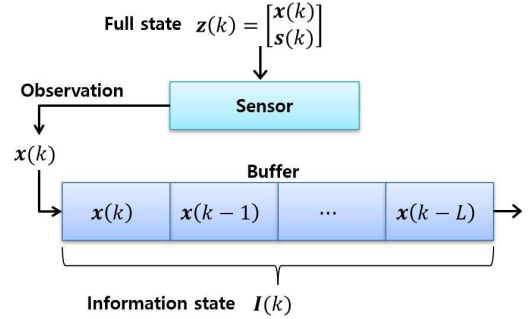,width=7cm}
\caption{An example of information state}\label{fig:information-state}
\end{figure}
In real applications, the full-state information is usually not
available. In this paper, we adopt the following
assumption.
\begin{assumption}\label{assumption1}
${\bf x}(k)$ is measured in real time, while ${\bf s}(k)$ cannot
be measured.
\end{assumption}

When the state is partially observable, the observation ${\bf
x}(k)$ loses the Markov property~\cite[pp.~63]{resnick2013adventures} in general, i.e., its transition usually
depends on all past history of ${\bf x}(k)$ and the current
action ${\bf u}(k)$. We formally define this property.
\begin{tcolorbox}[colframe=black,,colback=white,size=small]
\begin{definition}[Markov property]\label{def:MDP-property}
The process $({\bf I}(k),{\bf u}(k))_{k=1}^\infty$ is said
to satisfy the Markov property if
\begin{align*}
&{\bf I}(k+1) \sim {\mathbb P}[{\bf I}(k+1)=\cdot
|{\bf I}(k) = I(k),{\bf u}(k) = u(k)]\\
&={\mathbb P}[{\bf I}(k+1)= \cdot |{\bf I}(i)=I(i),i=0,\ldots,k,{\bf u}(k) = u(k)].
\end{align*}
\end{definition}
\end{tcolorbox}

Define the space of information ${\cal I}$ and let ${\bf
I}(k)\in {\cal I}$ be an available information at time $k$,
called the information state. In other words, ${\bf I}(k)$ is a
general form of the observation and can include artificially
crafted information from the pure observation ${\bf x}(k)$. For
example, the information is the current continuous state ${\bf x}(k)$,
then the information space is $X$. Another example of the information state is the finite memory information state.
\begin{tcolorbox}[colframe=black,,colback=white,size=small]
\begin{definition}[Finite memory information state]\label{def:finite-momory-information}
The finite memory information state at time $k$ is defined as
\begin{align*}
&{\bf I}(k)=:({\bf x}(k),{\bf x}(k-1),\ldots ,{\bf x}(k - L))\in {\cal I}= X^{L+1}.
\end{align*}
\end{definition}
\end{tcolorbox}

\cref{fig:information-state} illustrates its concept and how to implement it in practice. It is known that the finite memory information structure in~\cref{def:finite-momory-information}
can alleviate the problem related to the POMDP~\cite{mnih2015human}. In particular, roughly speaking, it can reduce the sensitivity of the next state's distribution with respect to the current state.

\subsection{Output-Feedback Policy}

A deterministic control policy $\pi: {\cal I} \to U$ is a map from
the information space $\cal I$ to the control space $U$. In this paper, the set
of all admissible control policies is denoted
by $\Pi$. In addition, the sequence of control policies
$(\pi_0,\pi_1,\ldots)\in \Pi^\infty$ is denoted by $\bar
\pi$. If $\pi_0=\pi_1= \ldots = \pi$, then $\bar\pi$ or $\pi$ is
called a stationary control policy. In this case,
$\bar\pi$ will be simply denoted by $\pi$ if there exists no confusion. For the MDP, the episode is defined as a single realization of the state-action-reward trajectory. 
\begin{definition}[Episode]\label{def:episod}
An episode of the MDP~\eqref{eq:stochastic-system} under any
policy $\bar \pi$ is defined as the process
$({\bf x}(k),{\bf s}(k),{\bf u}(k),{\bf r}({\bf x}(k),{\bf u}(k)))_{k =
0}^\tau$, where $\tau$ is the random stopping time.
\end{definition}

In~\cref{def:episod}, $\tau$ may be finite or not. In this paper, the stopping time is defined as the final time step before the first time instant that continuous state exists $X$.
\begin{assumption}[Stopping time]\label{assump:episodic}
The stopping time $\tau$ is defined as the first time step such that ${\bf
x}(\tau+1) \notin X$.
\end{assumption}

\cref{assump:episodic} is useful when we consider $X$ which is a
strict subset of ${\mathbb R}^n$, where $n$ is the dimension of
$X$. This is the case if we use a linear function
approximation~\cite{geramifard2013tutorial} which locally
approximates the value function or consider a compact $X$ to
guarantee the stability of approximate dynamic programming algorithms. Since the stopping time depends on the initial state ${\bf I}(0)$ and the policy $\bar \pi$, it will be denoted by $\tau ({\bf I}(0);\bar \pi )$ throughout the paper, while for brevity, $\tau$ will be used when it is clear from the context. When ${\bf x}(0) \notin X$, we set $\tau = -1$.

\subsection{Problem Statement}
For $I \in {\cal I}$, the value associated with a given $\bar\pi\in \Pi^\infty$ is defined as
\begin{align}
&J^{\bar\pi}(I): = {\mathbb E}\left[ \left. \sum_{i=0}^{\tau({\bf I}(0);\bar\pi)} {\gamma^i {\bf r}({\bf
x}(i),\pi_i ({\bf I}(i)))} \right|{\bf I}(0)=I \right],\label{J-alpha}
\end{align}
where $\tau({\bf I}(0);\bar\pi)$ is the stopping time given ${\bf I}(0)$ and $\bar
\pi$, and the expectation is taken with respect to the
episode. The optimal control design problem is stated as follows.
\begin{tcolorbox}[colframe=black,,colback=white,size=small]
\begin{problem}[Optimal Decision]
Consider the finite memory information state in~\cref{def:finite-momory-information}. Find $\pi$ such that
\begin{align*}
&\pi^*(I):=\arginf_{\pi\in \Pi} J^\pi(I),
\end{align*}
for all $I\in {\cal I}$.
\end{problem}
\end{tcolorbox}

In this paper, the optimal cost will be denoted by $J^*(I):= J^{\pi^*}(I)$.

\section{Dynamic Programming with Markov Property}\label{section:dynamic programming}

In this section, we study dynamic programming (DP) approaches to
find the optimal cost function and its corresponding stationary
control policy under the Markov assumption. In this section, we consider the finite memory information state in~\cref{def:finite-momory-information}.
\begin{assumption}\label{assump:Markov-property}
The process $({\bf I}(k),{\bf u}(k))_{k=0}^\tau$ satisfies the Markov
property.
\end{assumption}

In other words, $({\bf I}(k),{\bf u}(k))_{k=0}^\tau$ is determined based on an MDP with the transition density function $p_{\bf I} (I'|I,u)$ from the current information $I$ to $I'$
given $u$.
Under~\cref{assumption0}, the quantity~\eqref{J-alpha} is always
finite, and hence well defined. The first property of $J^*$ is its
boundedness on ${\cal I}$.
\begin{proposition}\label{thm:bound-on-J*}
$J^*\leq M/(1-\gamma)$ on ${\cal I}$.
\end{proposition}
\begin{proof}
By using the definition~\eqref{J-alpha} and~\cref{assumption0}, we
have $J^* (I) \le\sum_{i=0}^\infty {\gamma^i M}= M/(1-\gamma)$.
\end{proof}

Typical DP approaches~\cite{bertsekas1996neuro} convert~\cref{def:episod} into a fixed point problem of a mapping called the Bellman operator. For a given $\pi\in \Pi$, we also define the following Bellman operator:
\begin{align}
&(T_\pi J_0)(I):= R(x(0),\pi(I))\nonumber\\
&+\gamma {\mathbb E}[{\mathbb I}_X ({\bf x}(1)) J_0({\bf
I}(1))|{\bf I}(0)= I,{\bf
u}(0)=\pi(I)],\label{eq:T-pi-operator}
\end{align}
where ${\mathbb I}_X:X\to \{0,1\}$ is the
indicator function,
\begin{align*}
{\mathbb I}_X(x) = \begin{cases}
 1 \quad {\rm if} \quad x\in X\\
 0 \quad {\rm otherwise}\\
 \end{cases}
\end{align*}
and
\begin{align*}
{\bf I}(0)&= ({\bf x}(0),{\bf x}(-1),\ldots,{\bf x}(-L)),\\
{\bf I}(1)&= ({\bf x}(1),{\bf x}(0),\ldots,{\bf x}(-L+1)),\\
I&= (x(0),x(-1),\ldots,x(-L)),\\
I'&= (x(1),x(0),\ldots,x(-L+1)).
\end{align*}

In~\eqref{eq:T-pi-operator}, the indicator function is included to take into account the exit event. Then, the value function $J^\pi$ in~\eqref{J-alpha} corresponding to
$\pi$ satisfies $J^\pi = T_\pi J^\pi$, which is called the Bellman equation.
\begin{tcolorbox}[colframe=black,,colback=white,size=small]
\begin{theorem}\label{thm:Bellman-eq}
$J^\pi = T_\pi J^\pi$ holds.
\end{theorem}
\end{tcolorbox}

Similarly, for any bounded function $J_0:{\cal I}\to {\mathbb R}_+$, we can define
the static operator
\begin{align}
&(TJ_0)(I)\nonumber\\
&:=\inf_{u \in U} \{R(x(0),u)\nonumber\\
& + \gamma {\mathbb E}[{\mathbb I}_X ({\bf x}(1))J_0({\bf I}(1))|{\bf I}(0)=I,{\bf u}(0)=u]\},\nonumber\\
&:=\inf_{u \in U} \left\{R(x(0),u)+\gamma \int_{x(1)\in X} {J_0(I')p_{\bf x}
(x(1)|I,u)dx(1)}\right\},\label{eq:T-operator}
\end{align}

Define the space of bounded
functions $J:{\cal I}\to{\mathbb R}_+$ as ${\cal M}:=\{J:{\cal
I}\to {\mathbb R}_+:J <\infty \}$. It can be proved that $({\cal
M},d)$ is a complete metric space~\cite[pp.~301]{csuhubi2003preliminaries} with the metric
\begin{align*}
&d(J,J'):=\sup_{I\in {\cal I}}|J(I)-J'(I)|.
\end{align*}
In the following, we prove that the optimal cost $J^*$
uniquely satisfies $TJ^*= J^*$ called the Bellman's equation, and the
sequence, $(J_k)_{k=0}^\infty$, generated by the DP algorithm, $J_{k+1}=T J_k,J_0\equiv 0$ (called value iteration), converges to $J^*$ under~\cref{assumption0}. We note that all proofs of this paper are contained in Appendix of the online supplemental material~\cite{lee2018dynamic}.

\begin{tcolorbox}[colframe=black,,colback=white,size=small]
\begin{theorem}[Convergence]\label{thm:convergence-of-DP}
The sequence $(J_k)_{k=0}^\infty$ generated by the DP algorithm,
$J_{k+1}(I)=(TJ_k)(I)$, $I\in {\cal I}$ with $J_0\equiv 0$
uniformly converges to $J^*$ w.r.t. the metric $d$.
\end{theorem}
\end{tcolorbox}
\begin{remark}
For MDPs where continuous and discrete state-spaces coexist and
are coupled, a convergence result of the DP was
addressed in~\cite[Theorem~2,Theorem~3]{kolmanovsky2008discrete}.
However, the proof in~\cite{kolmanovsky2008discrete} cannot be
directly applied to our case as the MDP has continuous stochastic
disturbances for the system
in~\eqref{eq:stochastic-system}.
\end{remark}

If $J^*$ is known, then the optimal control policy can be
recovered by using
\begin{align}
&\pi^*(I):=\inf_{u\in U} \{R(x(0),u)\nonumber\\
&+\gamma {\mathbb E}[{\mathbb I}_X ({\bf x}(1))J^*({\bf
I}(1))|{\bf I}(0)=I,{\bf u}(0)=u]\}\label{eq:optimal-policy}
\end{align}
provided that the infimum is attained, and this is the case when
$U$ is a discrete and finite set or when $U$ is compact. A
disadvantage of the policy recovery form~\eqref{eq:optimal-policy}
lies in the fact that it relies on the model knowledge. A way to overcome this difficulty is to use the
so-called action-value function or
Q-function~\cite[pp.~192]{bertsekas1995dynamic},~\cite{watkins1989learning}.
In particular, for a given policy $\pi$ and corresponding
$J^\pi$, the Q-function, $Q^\pi:{\cal I} \times
U\to{\mathbb R}_+$, is defined as
\begin{align*}
&Q^\pi(I,u)\\
&:= R(x(0),u)+\gamma {\mathbb E}[{\mathbb I}_X ({\bf
x}(1))J^\pi({\bf I}(1))|{\bf I}(0)=I,{\bf u}(0)=u].
\end{align*}
Since $J^\pi = T_\pi J^\pi$, comparing the above equation
with~\eqref{eq:T-pi-operator}, one can prove $J^\pi (I) =
Q^\pi(I,\pi(I))$ and
\begin{align*}
&Q^\pi(I,u) = R(x(0),u)\\
&+\gamma {\mathbb E}[{\mathbb I}_X ({\bf x}(1))Q^\pi({\bf
I}(1),\pi({\bf I}(1)))|{\bf I}(0)=I,{\bf u}(0)=u].
\end{align*}
Similarly, the optimal Q-function is defined as
\begin{align}
&Q^*(I,u)\nonumber\\
&:= R(x(0),u)+\gamma {\mathbb E}[{\mathbb I}_X ({\bf
x}(1))J^*({\bf I}(1))|{\bf I}(0)=I,{\bf u}(0)=u].
\label{eq:Q-factor}
\end{align}
By comparing this definition with~\eqref{eq:optimal-policy}, the
optimal policy can be expressed as $\pi^*(I): = \arginf_{u'\in U}
Q^*(I,u')$. In addition, by using the definition of $Q$-function and
$J^*=TJ^*$, the optimal cost can be represented by $J^*(I) =
\inf_{u\in U} Q^*(I,u)$. Plugging it into~\eqref{eq:Q-factor}
yields
\begin{align}
&Q^*(I,u)=R(x(0),u)\nonumber\\
&+\gamma {\mathbb E}[ {\mathbb I}_X ({\bf x}(1))
\inf_{u'\in U} Q^*({\bf I}(1),u')|{\bf I}(0)=I,{\bf u}(0)u]\nonumber\\
&=:(FQ^*)(I,u)\label{eq:F-operator}
\end{align}
where the expectation is with respect to ${\bf I}(1)$. Then,~\eqref{eq:Q-factor} can be written as the Q-Bellman equation
$Q^*=FQ^*$, which is equivalent to the Bellman equation
$J^*=TJ^*$. The $Q$-value iteration, $Q_{k + 1}=FQ_k$ with $Q_0\equiv 0$,
generates sequence $(Q_k)_{k=0}^\infty$ that converges to $Q^*$
under the same condition as in the DP.
\begin{corollary}[Convergence]
The sequence, $(Q_k)_{k=0}^\infty$, generated by the DP algorithm,
$Q_{k+1}(I,u)=(FQ_k)(I,u)$, $I\in {\cal I}$, $u\in U$, with $Q_0\equiv 0$
uniformly converges to $Q^*$ w.r.t. the metric $d$.
\end{corollary}

An advantage of the $Q$-value iteration is that once found, the control
policy can be recovered without the model knowledge.
\begin{remark}
In practice, the value function or Q-function can be represented by universal
function approximators~\cite{bertsekas1996neuro}, for example, a deep neural network or radial basis functions. With such an approximator, the convergence proof should be modified for the specific
approximator. However, this topic is out of the scope of this paper. Moreover, to implement DP algorithms in this paper, one needs to integrate over the entire information space in the
definition of the Bellman operators, for instance, $F$ in~\eqref{eq:T-operator}. In practice, reinforcement learning algorithms~\cite{sutton1998reinforcement} can be applied to approximate the Bellman operators by their stochastic approximations, and perform the DP algorithms using the stochastic estimations.
\end{remark}

\section{Dynamic Programming without Markov Property}\label{section:dynamic programming2}
If we consider the POMDP, then the operators
in~\eqref{eq:T-operator} and~\eqref{eq:T-pi-operator} are not well
defined. In~\cite{singh1994learning}, the authors considered a
behavior policy $\pi_b$ and the corresponding limiting stationary
distribution of the MDP~\eqref{eq:stochastic-system}
\begin{align*}
&\mathop {\lim }_{k \to \infty } p_{\bf x} ( \cdot
|{\bf{x}}(k),{\bf s}(k),\pi_b ({\bf{I}}(k))) = \xi_{\bf x}(\cdot;\pi _b),\\
&\mathop {\lim }_{k \to \infty } p_{\bf s} ( \cdot
|{\bf{s}}(k),{\bf{x}}(k)) = \xi_{\bf s}( \cdot ;\pi _b),
\end{align*}
where $\xi_{\bf x}(\cdot;\pi _b)$ is the stationary distribution of the continuous state and $\xi_{\bf s}( \cdot ;\pi_b)$ is the stationary distribution of the discrete space under the behavior policy $\pi_b$ provided that they exist.

Then, the probability density of the next continuous state ${\bf
x}(k+1)$ given current state ${\bf
x}(k)=x$ and action ${\bf u}(k)=u$ is
\begin{align*}
&{\mathbb P}[{\bf x}(k + 1) = \cdot |{\bf{x}}(k) = x,{\bf u}(k) =
u]\\
& = \sum_{s \in S} {p_{\bf x} ( \cdot |x,s,u)\xi_{\bf
s}(s;\pi_b)}.
\end{align*}
Therefore, an information state transition density
function $p_{\bf I} (I'|I,u;\pi_b)$ is well defined, and the
operator corresponding to~\eqref{eq:F-operator} can be defined in
this setting. The DP solution under this condition is similar to how the standard Q-learning operates with POMDPs (see~\cite[Theorem~2]{singh1994learning}).

To consider a more generic scenario, a different
approach is adopted in this paper. For the analysis, the concept of the belief state is introduced first. The belief state~\cite{hauskrecht2000value}, ${\bf b}(k)\in
\Delta_{|S|}$, at time $k$ is defined as the probability
distribution of ${\bf z}(k)$ at time $k$.
\begin{tcolorbox}[colframe=black,,colback=white,size=small]
\begin{definition}[Belief state]
Consider the MDP~\eqref{eq:stochastic-system}. The belief state, $({\bf b}(k))_{k = 0}^\tau$, is defined by the recursion
\begin{align*}
&{\bf b}(k + 1) = P({\bf{x}}(k))^T {\bf{b}}(k),\quad b(0)= \rho_d.
\end{align*}
\end{definition}
\end{tcolorbox}

Note that $({\bf b}(k))_{k=0}^\tau$ is also a stochastic process due
to the randomness of ${\bf x}(k)$. In particular, each
${\bf b}(k)$ is defined on the probability space $(\Omega
,{\cal F},v)$ such that $\Omega = \Delta _{|S|}$, $v(F) =
{\mathbb P}[{\bf b}(k)\in F],\forall F \in {\cal F}$, and
${\cal F}$ is a $\sigma$-algebra on $\Delta _{|S|}$. The sequence $({\bf b}(k))_{k=0}^\tau$ corresponds to a single realization of $({\bf x}(k))_{k =
0}^\tau$ under a fixed policy $\pi$.

When the model and the current state are exactly known, then the next belief state can be computed at every time step in a deterministic fashion based on the current belief state. With the deterministic belief state propagation, a new DP is introduced in the next subsection.

\subsection{Known Belief State}
Assuming that the belief state is known, we consider the particular information structure in this subsection
\begin{align*}
{\bf\hat I}(k) =:({\bf x}(k),{\bf b}(k))\in {\cal\hat I}=X \times\Delta_{|S|}.
\end{align*}
Then, $({\bf\hat I}(k))_{k=0}^\tau$ is an MDP, i.e., its evolution can be
expressed as ${\bf\hat I}(k+1)\sim {\mathbb P}[{\bf\hat
I}(k+1)=\hat I(k+1)|{\bf\hat I}(k)=\hat I(k),{\bf u}(k)=u]$ as illustrated
in~\cref{fig:stochastic-system2}. 
\begin{figure}[h!]
\centering\epsfig{figure=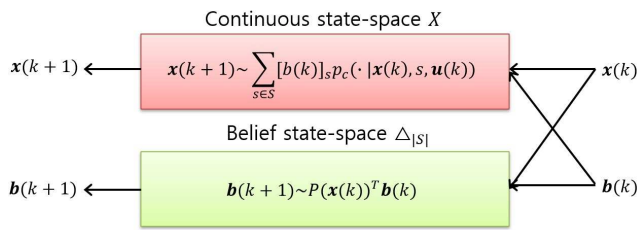,width=8.5cm} \caption{MDP
with belief state.}\label{fig:stochastic-system2}
\end{figure}
Therefore, by~\cref{thm:convergence-of-DP}, the DP
\begin{align}
&\hat Q_{k+1}(\hat I,u) = (\hat F \hat Q_k)(\hat I,u),\quad \hat Q_0(\hat I,u) \equiv 0,\label{eq:belief-DP}
\end{align}
converges to $\hat Q^*(\hat I,u)$, where
\begin{align}
&(\hat F \hat Q_0)(\hat I,u)= R(x,u)\nonumber\\
& + \gamma \int_{x' \in X} {\inf_{u' \in U}
Q_0(\hat I',u')\sum_{s \in S} {[b]_s} p_{\bf x}
(x'|x,s,u)dx'},\label{eq:app2:eq1}
\end{align}
$\hat I= (x,b)$, $\hat I'= (x',b')$, $[b]_s$ is the $s$th element of $b$, $b' = P(x)^T b$, and
$x'$ is the next state corresponding to $\hat I'$. The optimal solution of the DP~\eqref{eq:belief-DP} may give better performance compared to a DP solution, if exists, without the belief state information. In the next subsection, we introduce a DP-like algorithm without the belief state information, which may not have a fixed point solution. However, we establish a convergence of the algorithm to a set around the optimal solution $\hat Q^*$.

\subsection{Unknown Belief State}

In this subsection, we consider the case that the belief state is unknown. In this case, we define a sequence of operators $({\bf F}^{(k)})_{k=0}^\tau$ associated with a sequence of the belief states $({\bf b}(k))_{k=0}^\tau$.
\begin{tcolorbox}[colframe=black,,colback=white,size=small]
\begin{assumption}\label{assump:belief-state-sequence}
$({\bf b}(k))_{k=0}^\tau$ is a sequence belief states corresponding to a single realization of the episode under a fixed (behavior) policy $\pi$.
\end{assumption}
\end{tcolorbox}
\begin{tcolorbox}[colframe=black,,colback=white,size=small]
\begin{definition}
For any bounded $Q_0:{\cal I} \times U\to{\mathbb R}_+$, define a sequence of operators $({\bf F}^{(k)})_{k=0}^\tau$ associated with the sequence of belief states, $({\bf b}(k))_{k=0}^\tau$ in~\cref{assump:belief-state-sequence}, as
\begin{align}
&({\bf F}^{(k)} Q_0 )(I,u) = R(x,u)\nonumber\\
&+ \gamma \int_{x' \in X} \inf_{u' \in U} Q_0
(I',u')\nonumber\\
& \times\sum_{s\in S} {[{\bf b}(k)]_s } p_{\bf x}
(x(1)|x(0),s,u)dx(1),\label{eq:F-operator-non-Markov}
\end{align}
where
\begin{align*}
{\bf I}(0)&= ({\bf x}(0),{\bf x}(-1),\ldots,{\bf x}(-L)),\\
{\bf I}(1)&= ({\bf x}(1),{\bf x}(0),\ldots,{\bf x}(-L+1)),\\
I&= (x(0),x(-1),\ldots,x(-L)),\\
I'&= (x(1),x(0),\ldots,x(-L+1)).
\end{align*}

\end{definition}
\end{tcolorbox}

Note that the sequence of operators, $({\bf F}^{(k)})_{k=0}^\tau$, is stochastic as each ${\bf F}^{(k)}$ depends on ${\bf b}(k)$, and
the DP
\begin{align}
&{\bf Q}_{k+1}={\bf F}^{(k)} {\bf Q}_k,\quad  {\bf Q}_0 \equiv 0\label{eq:DP-POMDP}
\end{align}
may not converge in general as ${\bf F}^{(k)}$ is time-varying. However, under certain
conditions, we can obtain a bounded set around $\hat Q^*$ to which $({\bf
Q}_{k})_{k=0}^\tau$ converges as $\tau \to \infty$. For convenience, new notations are adopted.
Define
\begin{align}
\beta(I,b)&:={\mathbb P}[{\bf s}(0)=\cdot|{\bf I}(0)=
I,{\bf b}(-L)=b]\nonumber\\
&= P(x(0))^T \times\cdots \times P(x(-L))^T
b.\label{eq:beta-notation}
\end{align}

In~\eqref{eq:beta-notation}, $b\in \Delta_{|S|}$ represents the belief state at time
$-L$, and $\beta(I,b)$ implies the belief state at time $0$ given ${\bf I}(0) = I$ and ${\bf b}(-L) = b$. The following result establishes the fact that $({\bf
Q}_{k})_{k=0}^\tau$ converges to a bounded set around $\hat Q^*$ as $\tau \to \infty$.
\begin{theorem}\label{thm:Q-error-bound}
We assume that $\tau = \infty$. For any $L >0$, let $l_L\in {\mathbb R}_+$ be a Lipschtz constant
such that
\begin{align}
&\|\beta(I,b)-\beta(I,b') \|_\infty \le l_L \|b-b'
\|_\infty,\nonumber\\
&\forall b,b'\in \Delta_{|S|},\quad I\in X^{L+1}.\label{eq:Lipschitz-const}
\end{align}

Let $l_L^*$ be the infimum over all such constants $l_L$. Moreover, let $({\bf Q}_{k})_{k=0}^\infty$ be the sequence generated by the DP~\eqref{eq:DP-POMDP}. Let $\hat Q^*$ be the optimal
solution generated by the DP~\eqref{eq:belief-DP}. For all $k\geq 0$, the worst error bound
is given by
\begin{align*}
&\sup_{u \in U,I\in X^{L + 1} ,b \in \Delta
_{|S|} } |{\bf Q}_k (I,u) - \hat Q^* ((x(0),\beta(I,b)),u)|\\
& \le \frac{\gamma (1 - \gamma ^k )M |S|^2}{(1 - \gamma )^2
}l_L^* + \frac{\gamma ^k M}{1 - \gamma}
\end{align*}
with probability one, where $x(0)$ is the first element of the tuple
$I$, $b$ corresponds to the belief state at time $-L$. In the
limit $k \to \infty$, we obtain
\begin{align*}
&\limsup_{k \to \infty }\\
& \sup_{u \in U,I\in X^{L + 1} ,b \in \Delta _{|S|} } |{\bf{Q}}_k (I,u) -
\hat Q^* ((x(0),\beta (I,b)),u)|\\
&\le \frac{{M\gamma |S|^2 }}{{(1 - \gamma )^2 }}l_L^*
\end{align*}
with probability one.
\end{theorem}

In practice, the discount factor $\gamma \in [0,1)$ is close to
one; thus the error bound may be large. However, $l_L$ is often
small as in many applications~\cite{donghwan2018approximate}.
Then, under some conditions, we have $\lim_{L\to\infty} l_L^*= 0$.
As a simple example, assume that $P(x)$ is a constant matrix $P$
and that the Markov chain has a unique stationary distribution
$\mu$ such that $\mu^T P=\mu^T$. This implies that $\lim_{L\to
\infty}(P^L)^T b =\lim_{L\to\infty} (P^L)^T b'= \mu$ for any $b,
b'\in \Delta_{|S|}$. Therefore, $\lim_{L\to\infty} \| (P^L)^T b -
(P^L)^T b'\|_2=0$, meaning that $\lim_{L \to \infty} l_L^* = 0$.

\section*{Conclusion}
In this paper, we have studied a DP framework for POMDPs with
jointly continuous and discrete state-spaces. A finite observation history has
been used as an information structure of an output-feedback control policy. We have established a convergence of the DP algorithm to a set round an optimal solution. Developments of reinforcement learning algorithms based on the current analysis can be potential future research directions.

\bibliographystyle{IEEEtran}
\bibliography{reference}

\clearpage
\newpage

\onecolumn\

\appendix

Throughout this section, we use the notation
\begin{align*}
{\bf I}(0)&= ({\bf x}(0),{\bf x}(-1),\ldots,{\bf x}(-L)),\\
{\bf I}(1)&= ({\bf x}(1),{\bf x}(0),\ldots,{\bf x}(-L+1)),\\
I&= (x(0),x(-1),\ldots,x(-L)),\\
I'&= (x(1),x(0),\ldots,x(-L+1)).
\end{align*}

\subsection{Proof of~\cref{thm:Bellman-eq}}
By the definition, we have
\begin{align*}
J^{\pi}(I) &:= {\mathbb E}\left[ \left. \sum_{i=0}^{\tau({\bf I}(0);\pi)} \gamma^i {\bf r}({\bf
x}(i),\pi({\bf I}(i)))\right|{\bf I}(0)=I \right]\\
&= {\mathbb E}\left[ {\bf r}({\bf x}(0),\pi ({\bf I}(0))) + \left. \sum_{i=1}^{\tau({\bf I}(0);\pi)} {\gamma^i{\bf r}({\bf x}(i),\pi({\bf I}(i)))} \right|{\bf I}(0)=I \right]\\
&= R(x(0),\pi(I))+\gamma {\mathbb E}\left[\left. {\mathbb I}_X ({\bf x}(1))\sum_{i=1}^{\tau ({\bf I}(0);\pi )} {\gamma ^{i - 1} {\bf r}({\bf x}(i),\pi ({\bf I}(i)))} \right|{\bf I}(0)=I \right]\\
&= R(x(0),\pi(I))+\gamma {\mathbb E}\left[ \left. {\mathbb I}_X ({\bf x}(1))\sum_{i=1}^{\tau ({\bf I}(1);\pi)+ 1} {\gamma^{i-1} {\bf r}({\bf x}(i),\pi({\bf I}(i)))} \right|{\bf I}(0)=I \right]\\
&= R(x(0),\pi(I)) + \gamma {\mathbb E}\left[ \left. {\mathbb I}_X ({\bf x}(1))\sum_{i=0}^{\tau ({\bf I}(1);\pi )} \gamma^i {\bf r}({\bf x}(i+1),\pi ({\bf I}(i+1))) \right|{\bf I}(0) = I \right]\\
&= R(x(0),\pi(I)) + \gamma {\mathbb E}[\left. {\mathbb I}_X ({\bf x}(1))J^\pi({\bf I}(1)) \right|{\bf I}(0)=I ]\\
&= T_\pi J^\pi(I).
\end{align*}

\subsection{Proof of~\cref{thm:convergence-of-DP}}
\label{appendix:proof1} To prove~\cref{thm:convergence-of-DP}, one
needs to first prove that $J_{k} \in {\cal M}$ for all $k\in
{\mathbb N}$ and that $T$ is a contraction map.
\begin{tcolorbox}[colframe=black,,colback=white,size=small]
\begin{proposition}\label{thm:property-of-T}
Assume $J_0\equiv 0$. The following statements hold true:
\begin{enumerate}

\item $J_k\in {\cal M}$ for all $k \in {\mathbb N}$. Especially,
$J_k \leq M /(1-\gamma)$;

\item $T$ is a map from $\cal M$ to $\cal M$;

\item $T$ is a contraction map.
\end{enumerate}
\end{proposition}
\end{tcolorbox}
\begin{proof}
To prove~1), note that $J_1 \le M$ by~\cref{assumption0}.
Moreover, by the definition of $T$, $J_2 \le M + \gamma M$, and
repeating it yields $J_k\le M\sum_{i=0}^{k-1} {\gamma^i}\le
M\sum_{i=0}^\infty {\gamma^i}=\frac{M}{1-\gamma}$, implying $J_{k}
\in {\cal M}$ for all $k \in {\mathbb N}$. 2) is proved directly
from 1). To prove~3), consider any $J,J' \in {\cal M}$. Then, we
have $d(TJ,TJ') = \sup_{I\in {\cal I}} |TJ(I)-TJ'(I)|\le \sup_{I \in {\cal I},u\in U} \gamma
| {\mathbb E}[J({\bf I}(1)) -J'({\bf I}(1))|{\bf I}(0)=I,{\bf
u}(0)=u] |\le \gamma d(J,J')$. Therefore, $T$ is a
contraction map.
\end{proof}
\begin{tcolorbox}[colframe=black,,colback=white,size=small]
\begin{proposition}\label{thm:monotonicity}
Assume $J_0\equiv 0$. The following statements hold true:
\begin{enumerate}
\item $T$ is monotone, i.e., if $J\geq J'$, then $TJ\geq TJ'$;

\item For any given $\pi\in \Pi$, $T_{\pi}$ is monotone;

\item $(J_k)_{k=0}^\infty$ is a monotonically non-decreasing
sequence;

\item $(J_k)_{k=0}^\infty$ uniformly converges to a unique fixed
point $J_\infty \in {\cal M}$, i.e., $T J_\infty=J_\infty$, w.r.t.
the metric $d$.
\end{enumerate}
\end{proposition}
\end{tcolorbox}
\begin{proof}
To prove~1), assume $J\geq J'$ and recall the definition in~\eqref{eq:T-operator}
\begin{align*}
(TJ_0)(I)&:=\inf_{u \in U} \{R(x(0),u) + \gamma {\mathbb E}[{\mathbb I}_X ({\bf x}(1))J_0({\bf I}(1))|{\bf I}(0)=I,{\bf u}(0)=u]\},\nonumber\\
&:=\inf_{u \in U} \left\{R(x(0),u)+\gamma \int_{x(1)\in X} {J_0(I')p_{\bf x}
(x(1)|I,u)dx(1)}\right\},
\end{align*}

Then, for any $I\in {\cal I}$, we have 
\begin{align*}
&(TJ)(I)-(TJ')(I)\\
&\ge \inf_{u \in U} \left\{ R(x(0),u) + \gamma\int_{x(1)\in X} J(I')p_{\bf x} (x(1)|I,u)dx(1) - R(x(0),u)-\gamma \int_{x(1)\in X} {J'(I')p_{\bf x}(x(1)|I,u)dx(1)} \right\}\\
& = \gamma \inf_{u \in U} \left\{ \int_{x(1)\in X} {(J(I') - J'(I'))p_{\bf x} (x(1)|I,u)dx(1)} \right\}\\
&\ge 0,
\end{align*}
where the last inequality follows
from the hypothesis $J\geq J'$. This completes the proof of 1).
The statement 2) can be proved in a similar way, so omitted. The
proof of~3) is completed by an induction argument. Since
$J_0\equiv 0$ and $J_1(I)=(TJ_0)(I):=\inf_{u \in U} R(x(0),u)\ge
0$ by~\cref{assumption0}, $J_1 \geq J_0$ holds. By the
monotonicity of the operator $T$ in~1), we have $T^i J_1\ge T^i
J_0,\forall i \in {\mathbb N}_+$, meaning $J_{k+1}\ge J_k ,\forall
i \in {\mathbb N}$, which concludes the proof of~3).
By~\cref{thm:property-of-T}, $T:{\cal M} \to {\cal M}$ is a
contraction map on the complete metric space $({\cal M},d)$. By
the Banach fixed point
theorem~\cite[Theorem~5.6.1]{csuhubi2003preliminaries}, $(J_k)_{k
= 0}^\infty$ converges to a unique fixed point $J_\infty \in {\cal
M}$, i.e., $J_\infty=T J_\infty$, w.r.t. the metric $d$. The
convergence is uniform w.r.t. the metric $d$, which is proved
directly from the definition of the convergence in the metric
space as discussed in~\cite[pp.~301]{csuhubi2003preliminaries}.
\end{proof}

The last step for the proof of~\cref{thm:convergence-of-DP} is to
prove $J_\infty = J^*$, where $J_\infty:=\lim_{k\to \infty} J_k$ and the sequence $(J_k)_{k=0}^\infty$ is generated by the DP algorithm, $J_{k+1}=(TJ_k)$ with $J_0\equiv 0$. To this end, we need to prove that for any
given $\bar\pi\in \Pi^\infty$, we have $\lim_{k\to\infty} T_{\pi_0} T_{\pi_1}\cdots T_{\pi_k}J_0= J^{\bar \pi}$, which will be proved using two intermediate lemmas below.
\begin{tcolorbox}[colframe=black,,colback=white,size=small]
\begin{lemma}\label{thm:summation-form}
Assume $J_0\equiv 0$. For any given $\bar\pi\in\Pi^\infty$ and $k
\ge 1$, $T_{\pi_0} T_{\pi_1}\cdots T_{\pi_k}J_0$ is described as
\begin{align}
&T_{\pi_0} T_{\pi_1}\cdots T_{\pi_k}J_0(I)= {\mathbb E}\left[ \left. \sum_{i=0}^{\min\{\tau ({\bf
I}(0);\bar \pi),k\}} {\gamma^i {\bf r}({\bf x}(i),\pi_i({\bf
I}(i)))} \right| {\bf I}(0)=I \right]\label{eq:eq5}
\end{align}
for any $I \in {\cal I}$.
\end{lemma}
\end{tcolorbox}
\begin{proof}
The claim will be proved by an induction argument. Let $k = 0$.
Since $J_0 \equiv 0$, by the definition
in~\eqref{eq:T-pi-operator}, $T_{\pi_0}J_0(I)$ is given by
\begin{align*}
&T_{\pi_0}J_0(I)= {\mathbb E}[ {\bf r} ({\bf x}(0),\pi_0({\bf
I}(0)))|{\bf I}(0)=I]= {\mathbb E}\left[ \left. \sum_{i=0}^{\min \{ \tau ({\bf
I}(0);\bar\pi),0\}} \gamma^i {\bf r}({\bf x}(i),\pi_i({\bf
I}(i))) \right| {\bf I}(0)=I \right],
\end{align*}
where we use the fact that $\tau ({\bf
I}(0);\bar\pi) \geq 0$ as $I \in {\cal I}$ and $x(0) \in X$.

Now, for an induction argument, suppose for $k \ge 1$
\begin{align*}
&T_{\pi_0} T_{\pi_1}\cdots T_{\pi_{k-1}}J_0(I)= {\mathbb E}\left[ \left. \sum_{i=0}^{\min \{\tau ({\bf
I}(0);\bar\pi),k-1\}} {\gamma^i {\bf r}({\bf x}(i),\pi_i({\bf
I}(i)))} \right| {\bf I}(0)=I \right]
\end{align*}
holds. Then, shifting the time index of the control policy by one in the
above equation yields
\begin{align}
&T_{\pi_1} T_{\pi_2}\cdots T_{\pi_k} J_0(I)= {\mathbb E}\left[ \left. \sum_{i=0}^{\min \{\tau ({\bf
I}(0);\bar \pi_{1:\infty}),k-1\}} {\gamma^i {\bf r}({\bf
x}(i),\pi_{i+1}({\bf I}(i))} \right| {\bf I}(0)= I
\right],\label{eq:eq2}
\end{align}
where $\bar \pi_{1:\infty}:=(\pi_1,\pi_2,\ldots)$. We apply the
operator $T_{\pi_0}$ to~\eqref{eq:eq2} to have
\begin{align}
T_{\pi_0} T_{\pi_1} \cdots T_{\pi_k} J_0(I)&= {\mathbb E}\left[ \left.
{\bf r}(x(0),\pi_0(I)) + \gamma {\mathbb E}[{\mathbb I}_X({\bf x}(1)) T_{\pi_1} \cdots T_{\pi_k}J_0({\bf I}(1))] \right| {\bf I}(0)=I \right]\nonumber\\
&= {\mathbb E}\left[ {\bf r}(x(0),\pi_0(I)) + {\mathbb I}_X({\bf x}(1)) \left.  \times \sum_{i=0}^{\min\{\tau({\bf I}(1);\bar \pi_{1:\infty}),k-1\}} {\gamma^{i+1} {\bf r}({\bf x}(i+1),\pi_{i+1}({\bf
I}(i+1)))} \right| {\bf I}(0)=I\right],\label{eq:eq3}
\end{align}
where the second equation is obtained by~\eqref{eq:eq2} and $\tau(
{\bf I}(1);\bar\pi_{1:\infty})\geq -1$ is the first time instant
the trajectory ${\bf x}(k)$ starting from ${\bf x}(1)$ exits $X$
given ${\bf I}(1)$ and $\bar \pi_{1:\infty}$.

Note that ${\bf x}(1)$ is a random variable and $\tau({\bf I}(1) ;\bar
\pi_{1:\infty})=-1$ when ${\bf x}(1)\notin X$. In this case, we define $\sum_{i=1}^{-1}{\cdot} =0$. By conditioning on the stopping time $\tau({\bf I}(0);\bar\pi)$, the expectation
in~\eqref{eq:eq3} is expressed as
\begin{align}
&{\mathbb E}[ \left. {\bf r}(x(0),\pi_0(I)) \right| \tau
(I;\bar\pi)=0] {\mathbb P}[\tau (I;\bar\pi)=0]\nonumber\\
&+ {\mathbb E}\left[{\bf r}(x(0),\pi_0(I)) + \sum_{i=0}^{\min \{
\tau ({\bf I}(1);\bar \pi_{1:\infty}),k-1\}} {\gamma^{i+1} {\bf r}({\bf x}(i+1),\pi_{i+1}({\bf I}(i+1)))} \left| {\bf I}(0)=I,\tau(I;\bar\pi)\ge 1 \right. \right]{\mathbb P}[\tau(I;\bar\pi)\ge 1].\label{eq:eq4}
\end{align}
In the second expectation, $\tau(I;\bar\pi)\ge 1$ implies $\tau
({\bf I}(1);\bar \pi_{1:\infty}) \ge 0$. Noting that ${\mathbb
P}[\tau ({\bf I}(0);\bar \pi)=j]={\mathbb P}[\tau({\bf I}(1);\bar
\pi_{1:\infty})=j-1]$ for all $j\ge 1$, the quantity $\min\{\tau
({\bf I}(1);\bar\pi_{1:\infty}),k-1\}$ can be rewritten as
\begin{align*}
&\min\{\tau({\bf I}(1),\bar\pi_{1:\infty}),k-1 \}=\min\{\tau({\bf
I}(0),\bar \pi)-1,k-1\}.
\end{align*}
Plugging it into the original formulation in~\eqref{eq:eq4}
results in 
\begin{align*}
&{\mathbb E}[ \left. {\bf r}(x(0),\pi_0(I)) \right| \tau
(I;\bar\pi)=0 ] {\mathbb P}[\tau (I;\bar\pi)=0]\\
&+ {\mathbb E}\left[\left. {\bf r}(x(0),\pi_0(I)) + \sum_{i=0}^{\min \{
\tau ({\bf I}(0);\bar \pi_{1:\infty})-1,k-1 \}} {\gamma^{i+1} {\bf r}({\bf x}(i+1),\pi_{i+1}({\bf I}(i+1)))} \right| {\bf I}(0)=I,\tau(I;\bar\pi)\ge 1 \right]{\mathbb P}[\tau(I;\bar\pi)\ge 1]\\
&={\mathbb E}[ \left. {\bf r}(x(0),\pi_0(I)) \right| \tau
(I;\bar\pi)=0 ] {\mathbb P}[\tau (I;\bar\pi)=0]\\
&+ {\mathbb E}\left[\left. {\bf r}(x(0),\pi_0(I)) + \sum_{i=1}^{\min \{
\tau ({\bf I}(0);\bar \pi_{1:\infty}),k\}} {\gamma^{i} {\bf r}({\bf x}(i),\pi_{i}({\bf I}(i)))} \right| {\bf I}(0)=I,\tau(I;\bar\pi)\ge 1 \right]{\mathbb P}[\tau(I;\bar\pi)\ge 1]\\
&={\mathbb E}\left[ \left. \sum_{i=0}^{\min \{\tau ({\bf I}(0);\bar \pi_{0:\infty}),k\}} {\gamma^{i} {\bf r}({\bf x}(i),\pi_{i}({\bf I}(i)))} \right| {\bf I}(0)=I \right],
\end{align*}
which is the desired result.
\end{proof}
\begin{tcolorbox}[colframe=black,,colback=white,size=small]
\begin{lemma}\label{thm:operator-convergence}
Assume $J_0\equiv 0$. For any given $\bar\pi\in\Pi^\infty$, we
have $\lim_{k \to\infty} T_{\pi_0}T_{\pi_1}\cdots T_{\pi_k} J_0 =
J^{\bar\pi}$.
\end{lemma}
\end{tcolorbox}
\begin{proof}
Define $T_{\pi_0}T_{\pi_1}\cdots T_{\pi_k} J_0=:J_k^{\bar\pi}$.
Using~\cref{thm:summation-form} and following the proof
of~\cref{thm:property-of-T}, it is easy to prove that $J_k^{\bar
\pi }$ is bounded on $\cal I$, i.e., $J_k^{\bar\pi}\leq
M/(1-\gamma)$. Moreover, from the definition, $J_k^{\bar \pi }$ is
non-decreasing in $k$. Therefore, the point-wise limit $\lim_{k
\to\infty}J_k^{\bar\pi}=:J_\infty^{\bar\pi}$ exists. Noting that
the function inside the expectation operator is bounded, we apply
the dominated convergence theorem~\cite[Theorem~1.5.6]{durrett2010probability} to have 
\begin{align*}
J_\infty^{\bar\pi}(I)&= \lim_{k \to\infty}J_k^{\bar\pi}(I)=\lim_{k \to\infty} T_{\pi_0}T_{\pi_1}\cdots T_{\pi_k} J_0(I)\\
&=\lim_{k \to\infty} {\mathbb E}\left[ \left. \sum_{i=0}^{\min\{\tau ({\bf
I}(0);\bar \pi),k\}} {\gamma^i {\bf r}({\bf x}(i),\pi_i({\bf
I}(i)))} \right| {\bf I}(0)=I \right]\\
& = {\mathbb E}\left[ \left. \lim_{k \to\infty} \sum_{i=0}^{\min\{\tau ({\bf
I}(0);\bar \pi),k\}} {\gamma^i {\bf r}({\bf x}(i),\pi_i({\bf
I}(i)))} \right| {\bf I}(0)=I \right]\\
& = J^{\bar\pi}(I)
\end{align*}
for all $I \in {\cal I}$, where the third equality is due to the dominated convergence theorem and the last equality is from the definition of $J^{\bar\pi}$. Therefore, $J^{\bar\pi} = J_\infty^{\bar\pi}$, and the desired result is obtained.
\end{proof}

Now, we are in position to prove that $J_\infty = J^*$ holds with
$J_0\equiv 0$.
\begin{tcolorbox}[colframe=black,,colback=white,size=small]
\begin{proposition}\label{thm:optimality}
Assume $J_0\equiv 0$. Then, $J_\infty = J^*$.
\end{proposition}
\end{tcolorbox}
\begin{proof}
We follow the proof of~\cite[Prop.~2.1]{bertsekas1996neuro}.
Define $\pi_\infty$ as
\begin{align*}
&\pi_\infty(I):= \inf_{u \in U} \{ R(x(k),u) +\gamma {\mathbb
E}[J_\infty ({\bf I}(1))|{\bf I}(0)=I,{\bf u}(0)=u] \}.
\end{align*}
The above quantity is well defined because $J_\infty\in {\cal M}$
by 4) of~\cref{thm:monotonicity} and~\cref{assumption0}. From the
definition~\eqref{eq:T-pi-operator}, it also implies
$T_{\pi_\infty} J_0 = T J_\infty = J_\infty$. Then, since $J_0 \le
J_\infty$ by the monotonicity in~\cref{thm:monotonicity}, we have
$J^{\pi_\infty}= \lim_{k \to \infty} T_{\pi_\infty}^k J_0 \le
\lim_{k\to \infty}T_{\pi_\infty}^k J_\infty =J_\infty$, implying
$J^{\pi_\infty}\leq J_\infty$,
where~\cref{thm:operator-convergence} is used in the first
equality, the last equality follows from $T_{\pi_\infty}J_\infty =
J_\infty$, and the inequality comes from $J_0\le J_\infty$ and the
monotonicity of~$T_{\pi_\infty}$ in~\cref{thm:monotonicity}. On
the other hand, by the definition and the monotonicity of $T$, we
have that for any policy $\bar\pi=(\pi_0,\pi_1,\ldots)$, $J_\infty
= \lim_{k\to\infty} T^k J_0 \le\lim_{k\to\infty} T_{\pi_0} T_{\pi
_1}\cdots T_{\pi_k} J_0 = J^{\bar \pi}$, meaning $J_\infty \leq
J^{\bar\pi}$ for all $\bar\pi\in\Pi^\infty$, where the last
equality follows from~\cref{thm:operator-convergence}. Combining
both inequalities results in $J^{\pi_\infty}\le J_\infty\le
J^{\bar\pi},\forall\bar\pi\in\Pi^\infty$, yielding $J^{\pi_\infty}
= J_\infty=J^*$.
\end{proof}
Now,~\cref{thm:convergence-of-DP} is directly proved by
using~\cref{thm:monotonicity} and~\cref{thm:optimality}.

{\em Proof of~\cref{thm:convergence-of-DP}}: By (4) of~\cref{thm:monotonicity}, $(J_k)_{k=0}^\infty$ uniformly converges to a unique fixed point $J_\infty \in {\cal M}$, i.e., $T J_\infty=J_\infty$, w.r.t. the metric $d$. By~\cref{thm:optimality}, $J_\infty = J^*$, and hence, $(J_k)_{k=0}^\infty$ uniformly converges to $J^*$ w.r.t. the metric $d$. This completes the proof.

\subsection{Proof of~\cref{thm:Q-error-bound}}\label{appendix:proof2}
To prove~\cref{thm:Q-error-bound}, we first introduce the
following basic relation lemma.
\begin{tcolorbox}[colframe=black,,colback=white,size=small]
\begin{lemma}[Basic relation]\label{lemma:basic-relation}
Let $({\bf Q}_{k})_{k=0}^\infty$ be the sequence generated by the
DP 
\begin{align*}
&{\bf Q}_{k+1}={\bf F}^{(k)} {\bf Q}_k,\quad {\bf Q}_0 \equiv 0, 
\end{align*}
where ${\bf F}^{(k)}$ is defined in~\eqref{eq:F-operator-non-Markov}. Let $\hat Q^*$ be the optimal
solution of $\hat Q^*=\hat F \hat Q^*$, where $\hat F$ is defined
in~\eqref{eq:app2:eq1}. Then, we have
\begin{align*}
&\sup_{u \in U,I \in X^{L+1},b \in \Delta_{|S|}} |{\bf Q}_{k+1}(I,u)- \hat Q^* ((x(0),\beta(I,b)),u)|\\
&\le\frac{\gamma l_L^* M|S|^2}{1-\gamma}+ \gamma\sup_{u \in U, I\in X^{L+1},b \in
\Delta_{|S|}} |{\bf Q}_k (I,u) - Q^* ((x(0),\beta (I,b)),u)|
\end{align*}
for all $k \geq 0$ with probability one, where $x(0) \in X$ is the
first element of $I$.
\end{lemma}
\end{tcolorbox}
\begin{proof}
Using the definitions of ${\bf F}^{(k)}$ and $\hat F$, we have
\begin{align*}
{\bf Q}_{k+1}(I,u)-\hat Q^* ((x(0),\beta(I,b)),u)&= ({\bf F}^{(k)} {\bf Q}_k)(I,u) - (\hat F \hat Q^*)((x(0),\beta(I,b)),u)\\
&= \gamma \int_{x(1)\in X} {\inf_{u'\in U} {\bf Q}_k (I',u')\sum_{s\in S} {[{\bf b}(k)]_s} p_{\bf x} (x(1)|x(0),s,u)dx(1)}\\
&- \gamma \int_{x(1) \in X} {\inf_{u' \in U} \hat Q^*
((x(1),\beta (I',b')),u')\sum_{s \in S} {[\beta (I,b)]_s }
p_{\bf x} (x(1)|x(0),s,u)dx(1)},\\
&= \gamma \int_{x(1)\in X} {\inf_{u'\in U} {\bf Q}_k(I',u')\sum_{s\in S} {[\beta(I,{\bf b}(k-L))]_s}
p_{\bf x} (x(1)|x(0),s,u)dx(1)}\\
&- \gamma \int_{x(1) \in X} {\inf_{u' \in U} \hat Q^*((x(1),\beta (I',b')),u')\sum_{s \in S} {[\beta (I,b)]_s }
p_{\bf x} (x(1)|x(0),s,u)dx(1)},
\end{align*}
where $b' = P(x)^T b$. Adding and subtracting
\begin{align*}
&\gamma \int_{x(1) \in X} {\inf_{u' \in U} {\bf Q}_k (I',u')\sum_{s \in S} {[\beta (I,b)]_s } p_{\bf x} (x(1)|x(0),s,u)dx(1)}
\end{align*}
to the last equation, we have
\begin{align*}
&{\bf Q}_{k+1}(I,u)-\hat Q^*((x,\beta(I,b)),u)\\
&\le \gamma |S| \int_{x(1) \in X} \inf_{u' \in U}
{\bf Q}_k (I',u') \times \sum_{s \in S} {|[\beta (I,{\bf b}(k-L))]_s - [\beta(I,b)]_s|} \frac{1}{|S|}p_{\bf x} (x(1)|x(0),s,u)dx(1)\\
&+ \gamma \int_{x(1)\in X} {\sup_{u' \in U} (|{\bf Q}_k
(I',u') - \hat Q^* ((x(1),\beta (I',b')),u'))} \times \sum_{s \in S} {[\beta(I,b)]_s} p_{\bf x}(x(1)|x(0),s,u)dx(1)\\
&\le\gamma |S|\int_{x(1)\in X} {\frac{M |S|}{1-\gamma}\| \beta(I,{\bf b}(k-L)) - \beta(I,b)\|_\infty} \times \sum_{s \in S} {\frac{1}{|S|}p_{\bf x} (x(1)|x(0),s,u)} dx(1)\\
&+ \gamma \int_{x(1) \in X} { \sup_{u' \in U} |{\bf Q}_k(I',u')-\hat Q^*((x(1),\beta (I',b')),u')|} \times \sum_{s \in S} {[\beta(I,b)]_s} p_{\bf x}(x(1)|x(0),s,u)dx(1),
\end{align*}
where in the second inequality, we use the definition of $\|\cdot
\|_\infty$ and the bound ${\bf Q}_k (I',u') \le M/(1-\gamma),\forall I' \in {\cal I},u'\in U$.
Using~\eqref{eq:Lipschitz-const}, we have
\begin{align*}
&{\bf Q}_{k+1}(I,u)-\hat Q^* ((x(0),\beta(I,b)),u)\\
&\le \frac{\gamma M|S|^2}{1-\gamma}l^*_L \| {\bf b}(k-L)-b\|_\infty \int_{x(1)\in X} {\sum_{s\in S}
{\frac{1}{|S|}p_{\bf x} (x(1)|x(0),s,u)} dx(1)}\\
& + \gamma \sup_{u'\in U,I' \in X^{L+1}}
|{\bf Q}_k (I',u')-\hat Q^*((x(1),\beta(I',b')),u')|\\
&\le \gamma \frac{M|S|^2}{1 - \gamma}l^*_L +\gamma \sup_{u' \in U,I' \in X^{L + 1} } (|{\bf
Q}_k (I',u') - \hat Q^* ((x(1),\beta (I',b')),u')|),
\end{align*}
where in the last inequality, we use the fact that $\| {\bf b}(k-L)-b\|_\infty \leq 1$ for any ${\bf b}(k-L), b$ inside the unit simplex $\Delta_{|S|}$ and $\sum_{s \in S} {\frac{1}{|S|}p_{\bf x}
(x(1)|x(0),s,u)}$ is a probability density function. The desired result
follows by taking the superimum over $I \in {\cal I},u\in U$, and $b
\in \Delta_{|S|}$ on the left-hand side of the last inequality.
\end{proof}

{\em Proof of~\cref{thm:Q-error-bound}}: Combining the
inequalities in~\cref{lemma:basic-relation} for $k=0$ and $k=1$ yields
\begin{align*}
&\frac{1}{\gamma^2}\sup_{u \in U,I\in
X^{L+1},b \in \Delta_{|S|} } |{\bf Q}_2(I,u)-\hat Q^*((x(0),\beta(I,b)),u)|\\
&\le (1+\gamma^{-1} )\frac{M|S|^2}{1-\gamma}l_L^*+\sup_{u \in U,I\in X^{L+1},b \in \Delta_{|S|} } |{\bf Q}_0(I,u) - \hat Q^* ((x(0),\beta (I,b),u)|.
\end{align*}
Repeating this $k-2$ times, one gets
\begin{align*}
&\sup_{u\in U,I \in X^{L + 1} ,b \in \Delta_{|S|} } |{\bf Q}_k (I,u) - \hat Q^* ((x(0),\beta(I,b)),u)|\\
&\le \gamma^k \left( \sum_{i=0}^{k - 1}
{\frac{1}{\gamma^i}} \right)\frac{M|S|^2}{1-\gamma}l_L^*+ \gamma^k \sup_{u \in U,I\in X^{L+1},b
\in \Delta_{|S|}} |{\bf Q}_0(I,u)-\hat Q^*((x(0),\beta(I,b),u)|.
\end{align*}
Using ${\bf Q}_0(I,u) \equiv 0$, $\hat Q^*((x(0),\beta (I,b),u)
\le M/(1-\gamma)$, and $\sum_{i=0}^{k-1} {(1/\gamma^i)}=(1-(\gamma^{-1})^k)/(1-\gamma^{-1})$, the first
inequality in~\cref{thm:Q-error-bound} follows. The second
inequality can be obtained by taking $\limsup_{k\to\infty}$ on
both sides of the inequality. This completes the proof.

\end{document}